\documentclass{article}
\usepackage[utf8]{inputenc}
\usepackage{amsthm}
\usepackage{amssymb}
\usepackage{amsmath}
\usepackage{epsfig}
\usepackage{colonequals}
\usepackage{enumerate}
\usepackage{enumitem}
\usepackage{hyperref}

\renewcommand{\AA}{\Gamma}
\newcommand{\PP}{{\mathcal P}}
\newcommand{\QQ}{{\mathcal Q}}

\newcommand\zz{\mathbb{Z}_3}
\newcommand\g{\mathbb{G}}

\newcommand\so{\overleftrightarrow}
\let\phi\varphi
\newtheorem{theorem}{Theorem}[section]
\newtheorem{corollary}[theorem]{Corollary}
\newtheorem{lemma}[theorem]{Lemma}
\newtheorem{observation}[theorem]{Observation}

\newtheorem{conjecture}[theorem]{Conjecture}
\newtheorem{problem}[theorem]{Problem}

\begin{document}
\title{On density of $\zz$-flow-critical graphs}
\author{%
     Zden\v{e}k Dvo\v{r}\'ak\thanks{Computer Science Institute (CSI) of Charles University,
           Malostransk{\'e} n{\'a}m{\v e}st{\'\i} 25, 118 00 Prague,
           Czech Republic. E-mail: \protect\href{mailto:rakdver@iuuk.mff.cuni.cz}{\protect\nolinkurl{rakdver@iuuk.mff.cuni.cz}}.
           Supported by project 22-17398S (Flows and cycles in graphs on surfaces) of Czech Science Foundation.}
     \and
     Bojan Mohar\thanks{Department of Mathematics, Simon Fraser University, Burnaby, B.C. V5A 1S6.
           E-mail: {\tt mohar@sfu.ca}.
           On leave from IMFM \&\ FMF, Department of Mathematics, University of Ljubljana.}
           \thanks{Supported in part by the NSERC Discovery Grant R611450 (Canada),
           and by the Research Project J1-2452 of ARRS (Slovenia).}%
}

\date{\today}
\maketitle

\begin{abstract}
For an abelian group $\AA$, a graph $G$ is said to be $\AA$-flow-critical if $G$ does not admit a nowhere-zero $\AA$-flow, but for each edge $e\in E(G)$, the contraction $G/e$ has a nowhere-zero $\AA$-flow.
We obtain a bound on the density of $\mathbb{Z}_3$-flow-critical graphs
drawn on a fixed surface, generalizing the planar case of the bound on the density of 4-critical graphs by Kostochka and Yancey.
\end{abstract}

\section{Overview}

Critical graphs play an important role in graph coloring theory: A graph $G$ is \emph{$c$-critical} if $\chi(G)=c$,
but every proper subgraph of $G$ is $(c-1)$-colorable.  Hence, a graph is $k$-colorable if and only if it does not
contain a $(k+1)$-critical subgraph.  In particular, interesting results regarding the chromatic number of
graphs on surfaces can be obtained by combining lower bounds on the density of critical graphs with the upper bounds
on the density of graphs of given genus arising from the generalized Euler's formula.

As an example, Gallai~\cite{galfor2}
proved that for $c\ge 3$, every $c$-critical graph different from $K_c$ has average degree at least $c-1+\frac{c-3}{c^2-3}$,
and in particular, any $4$-critical graph other than $K_4$ has average degree at least $3+\frac{1}{13}$.  On the other
hand, a graph of girth at least six with $n$ vertices drawn on a surface $\Sigma$ of Euler genus $\gamma$ has average degree at most $3+3(\gamma-2)/n$.
Hence, if such a graph is $4$-critical, it has at most $39(\gamma-2)$ vertices; and consequently, the 3-colorability of graphs
of girth at least six drawn on $\Sigma$ is characterized by a finite number of forbidden subgraphs.  Analogously,
the same is true for 4-colorability of triangle-free graphs and for 6-colorability.

Let us remark that Gallai's bound was subsequently improved in~\cite{krivelevich1997minimal,kostochka2003new}, and finally an asymptotically tight bound
was given by Kostochka and Yancey~\cite{koyanore}: Any $c$-critical graph with $n$ vertices has average degree at least $\frac{(c+1)(c-2)}{c-1}-O(1/n)$.

Tutte~\cite{tutteflow} famously observed that coloring is dual to nowhere-zero flows.
To state the result, we need several definitions, which we present in a slightly more general setting that will be useful later.
We allow graphs to have parallel edges, but no loops.
Let $\AA$ be an abelian group\footnote{We will implicitly assume that $\AA$ is nontrivial, i.e. that it contains at least 2 elements although most of the results hold also for the one-element group.}.  A \emph{$\AA$-boundary} for a graph $G$ is a function $\beta:V(G)\to\AA$ such that
$\sum_{v\in V(C)} \beta(v)=0$ for every component $C$ of $G$.  
If $\beta$ is a $\AA$-boundary for $G$, we say that the pair $\g=(G,\beta)$ is a \emph{$\AA$-bordered graph} (or a \emph{graph with boundary}). 
A \emph{symmetric orientation} $\so{G}$ of $G$
is the directed graph obtained by replacing every edge $e=uv$ of $G$ by a pair of oppositely directed edges
$e_u$ (directed towards $u$) and $e_v$ (directed towards $v$).  A function $f:E(\so{G})\to\AA$
is \emph{antisymmetric} if for every $e=uv\in E(G)$, $f(e_u)=-f(e_v)$.  A \emph{flow} in $(G,\beta)$
is an antisymmetric function $f:E(\so{G})\to\AA$ such that
for every $u\in V(G)$, $$\sum_{e=uv\in E(G)} f(e_v)=\beta(u).$$
The flow is \emph{nowhere-zero} if it does not use the value $0$.
By a \emph{$\AA$-flow} in a graph $G$, we mean a flow in $(G,0)$, where $0$ is the
$\AA$-boundary assigning to each vertex the value $0$.

\begin{theorem}[Tutte~\cite{tutteflow}]\label{thm-tutteflow}
Let $k$ be a positive integer and let $\AA$ be an abelian group of order $k$.  A connected plane graph $G$ is $k$-colorable
if and only of its dual graph $G^\star$ has a nowhere-zero $\AA$-flow.
\end{theorem}

Note that only the size of the group matters, and indeed, Tutte proved that if $\AA$ and $\AA'$ are abelian groups of the
same order, then the number of nowhere-zero $\AA$-flows is equal to the number of nowhere-zero $\AA'$-flows in any graph.
Tutte also stated a number of influential conjectures on the existence of nowhere-zero flows, generalizing the Five Color Theorem,
the Four Color Theorem, and Grötzsch theorem, respectively.

\begin{conjecture}[Tutte]
Let $G$ be a $2$-edge-connected graph.
\begin{itemize}[leftmargin=3.8cm]
\item[\bf{[$5$-flow Conjecture]}] $G$ has a nowhere-zero $\mathbb{Z}_5$-flow.
\item[\bf{[$4$-flow Conjecture]}] If $G$ does not contain Petersen graph as a minor, then $G$ has a nowhere-zero $\mathbb{Z}_4$-flow.
\item[\bf{[$3$-flow Conjecture]}] If $G$ is $4$-edge-connected, then $G$ has a nowhere-zero $\zz$-flow.
\end{itemize}
\end{conjecture}

While these conjectures are still open, Seymour~\cite{seymour1981nowhere} proved existence of nowhere-zero $\mathbb{Z}_k$-flows for every $k\ge 6$,
Robertson, Seymour, and Thomas~\cite{RSTtutte} claimed the proof of existence of nowhere-zero $\mathbb{Z}_4$-flows in subcubic graphs
with no Petersen minor (with one of the basic cases later appearing in~\cite{doublecross} and another one still unpublished~\cite{cubapex}),
and Lov{\'a}sz et al.~\cite{ltwz} proved existence of nowhere-zero $\zz$-flows in 6-edge-connected graphs.

Given the importance of critical graphs for coloring, it is natural to consider the dual concept of \emph{flow-critical graphs},
with the definition based on the observation that edge removal in a plane graph corresponds to contraction of the corresponding edge
in the dual graph.  

For a partition $\PP$ of the vertex set of a graph $G$, let $G/\PP$ be the graph obtained by identifying the vertices in each part of $\PP$ to a single
vertex and then removing all resulting loops.  If $\beta$ is a $\AA$-boundary for $G$, then let $\beta/\PP$ be the $\AA$-boundary for $G/\PP$ obtained by setting $(\beta/\PP)(p)=\sum_{v\in P} \beta(v)$ for each $P\in \PP$ identified to the vertex $p$.  For a $\AA$-bordered graph $\g=(G,\beta)$, we define $\g/\PP=(G/\PP,\beta/\PP)$.
For a set $B\subseteq V(G)$, we define $G/B$ and $\g/B$ as $G/\PP_B$ and $\g/\PP_B$ for the partition $\PP_B$ consisting of $B$ and single-vertex sets $\{v\}$ for $v\in V(G)\setminus B$.
If each part of $\PP$ induces a connected subgraph of $G$, we say that $\PP$ is \emph{$G$-connected}, and that
$G/\PP$ (or $\g/\PP$) is a \emph{contraction} of $G$ (or $\g$).
The partition $\PP$ is \emph{non-trivial} if at least one part contains more than one vertex, and in this case we call
the contraction \emph{proper}.
We say a $\AA$-bordered graph $\g$ is \emph{flow-critical} if $\g$ is connected and $\g$ has no nowhere-zero flow, but
every proper contraction of $\g$ has a nowhere-zero flow.  We say that a graph $G$ is \emph{$\AA$-flow-critical} if
the $\AA$-bordered graph $(G,0)$ is flow-critical.  Clearly, a graph $H$ has a nowhere-zero $\AA$-flow if and only
if no component of $H$ has a $\AA$-flow-critical contraction.

Given the results outlined above, $K_2$ is the only $\AA$-flow-critical graph when $|\AA|\ge 6$, and Tutte's 5-flow conjecture is equivalent
to the claim that there are no $\mathbb{Z}_5$-flow-critical graphs other than $K_2$.  Hence, the notion of $\AA$-flow-criticality is interesting only
for $\AA\in\{\zz,\mathbb{Z}_4,\mathbb{Z}_2^2\}$ (and possibly $\AA=\mathbb{Z}_5$).  The study of flow-critical graphs was started by Nunes da Silva and Lucchesi~\cite{da2008flow},
who showed some of their basic properties.  Recently, Li et al.~\cite{li20223} considered the extremal question of the density of $\zz$-flow-critical
graphs, which is also the subject of this paper.

As observed by Li et al.~\cite{li20223}, dualizing the results of Kostochka and Yancey~\cite{koyanore} shows that if $G$ is a planar $\zz$-flow-critical
graph, then $|E(G)|\le \frac{5|V(G)|-8}{2}$.  Li et al.~\cite{li20223} also demonstrated that this upper bound cannot be extended to non-planar graphs:
For $n\ge 7$, the graph $K_{3,n-3}^+$ obtained from the complete bipartite graph $K_{3,n-3}$ by adding an edge joining two of the vertices of degree $n-3$
is $\zz$-flow-critical, and it has $3n-8$ edges.  They conjectured that this bound is best possible.

\begin{conjecture}\label{conj-gen}
Every $\zz$-flow-critical graph $G$ satisfies $|E(G)|\le 3|V(G)|-5$.  Moreover, if $|V(G)|\ge 7$, then $|E(G)|\le 3|V(G)|-8$.
\end{conjecture}

Let us remark that Li et al.~\cite{li20223} only state the version of the conjecture for graphs with at least seven vertices,
but the version without this restriction follows by a straightforward inspection of graphs with at most $6$ vertices.
Li et al.~\cite{li20223} proved a weakening of this conjecture.

\begin{theorem}[Li et al.~\cite{li20223}]\label{thm-li}
Every $\zz$-flow-critical graph $G\neq K_2$ satisfies $|E(G)|\le 4|V(G)|-10$.
\end{theorem}

The \emph{Euler genus} of the surface obtained from the plane by adding $h$ handles and $c$ crosscaps is $2h+c$,
and the \emph{Euler genus} $g(G)$ of a graph $G$ is the minimum Euler genus of surfaces in which $G$ can be drawn without crossings.  
Note that the graph $K_{3,n-3}^+$ is far from being planar; indeed, it has Euler genus $\lceil \tfrac{n-5}{2} \rceil$ (see Ringel~\cite{Ringelbip}).
Given that the bound for planar graphs is substantially smaller than the bound for general graphs in Conjecture \ref{conj-gen},
it is natural to ask whether a better bound holds when the genus is bounded.  As our main result, we prove that this is indeed the case.

\begin{theorem}\label{thm-main}
If $G$ is a $\zz$-flow-critical graph, then
$$|E(G)|\le \frac{5|V(G)|+5g(G)-8}{2}.$$
\end{theorem}

Theorem~\ref{thm-main} implies Conjecture~\ref{conj-gen} for graphs $G$ whose Euler genus is at most $(|V(G)|-8)/5$,
and shows that for any fixed surface $\Sigma$, the conjecture can be verified for graphs drawn in $\Sigma$ by
examining a finite number of graphs.  Let us remark that there exist infinitely many planar $\zz$-flow-critical graphs $G$
satisfying $|E(G)|=\frac{5|V(G)|-8}{2}$, and thus the bound in Theorem~\ref{thm-main} cannot be improved to $c|V(G)|+f(g(G))$
for $c<5/2$ and any function $f$.  On the other hand, the dependence on genus is likely not the best possible.

\begin{problem}
What is the smallest constant $\alpha$ such that every $\zz$-flow-critical graph $G$ satisfies
$$|E(G)|\le \frac{5}{2}|V(G)|+\alpha g(G)+O(1)?$$
\end{problem}

Theorem~\ref{thm-main} and the example of $K_{3,n-3}^+$ show that $1\le \alpha\le 5/2$.

Our proof technique is a variation on the potential method argument of Kostochka and Yancey~\cite{koyanore},
and it naturally leads us to work in a more general setting of graphs with boundary.  Hence, we in fact prove
the following more general version of Theorem~\ref{thm-main}.

\begin{theorem}\label{thm-main2}
Let $\g=(G,\beta)$ be a $\zz$-bordered graph.  If $\g$ is flow-critical,
then
$$|E(G)|\le \frac{5|V(G)|+5g(G)-8}{2}.$$
\end{theorem}

For an abelian group $\AA$, a graph $G$ is said to be \emph{$\AA$-connected} if $(G,\beta)$ has a nowhere-zero flow
for every $\AA$-boundary $\beta$.  
Many of the results on the existence of nowhere-zero flows are known to be true
in the group-connectivity setting as well; for instance, 6-edge-connected graphs are $\zz$-connected~\cite{ltwz}.
Moreover, the following group-connectivity analogue to the 3-flow conjecture was proposed.

\begin{conjecture}[Jaeger et al.~\cite{jaeger1992group}]\label{conj-3con}
Every $5$-edge-connected graph is $\zz$-connected.
\end{conjecture}

Let us remark that the 3-flow conjecture can be reduced to 5-edge-connected graphs~\cite{kochol2001equivalent},
and thus it is implied by Conjecture~\ref{conj-3con}.  By Theorem~\ref{thm-main2}, every flow-critical
$\zz$-bordered graph of Euler genus at most one has average degree less than five.  Since every contraction of
a 5-edge-connected graph has minimum degree at least five, it follows that Conjecture~\ref{conj-3con} holds for planar and
projective-planar graphs; this was previously proved in~\cite{richter2016group,de2020strong}.
Similarly, Theorem~\ref{thm-main2} implies Conjecture~\ref{conj-3con} for graphs that can be made planar by
removal of at most three edges (even in the stronger setting where we can prescribe the flow on the removed
edges arbitrarily).

As our second result, we show that Conjecture~\ref{conj-3con}
implies the first part of Conjecture~\ref{conj-gen}.

\begin{theorem}\label{thm-smallcex}
Let $G$ be a graph with the smallest number of vertices satisfying the following conditions:
\begin{itemize}
\item There exists a $\zz$-boundary $\beta$ such that $(G,\beta)$ is flow-critical, and
\item $|E(G)|>3|V(G)|-5$.
\end{itemize}
Then $G$ is triangle-free, $5$-edge-connected, and the only $5$-edge-cuts in $G$ are the vertex neighborhoods.
\end{theorem}

The rest of the paper is organized as follows. In Section~\ref{sec-prelim}, we present basic auxiliary results.
In Section~\ref{sec-main}, we prove a strengthening of Theorem~\ref{thm-main2}, which includes the characterization
of the graphs for which the bound is tight.  In Section~\ref{sec-cex}, we prove Theorem~\ref{thm-smallcex}.

\section{Preliminaries}\label{sec-prelim}

\subsection{Flow-critical graphs}

We use the following well-known facts concerning flow-critical graphs.

\begin{observation}\label{obs-basic}
Let $\AA$ be an abelian group and let $\g=(G,\beta)$ be a $\AA$-bordered graph.
If $\g$ is flow-critical, then
\begin{itemize}
\item $G$ is a 2-connected graph or $K_2$,
\item if $\AA\neq\mathbb{Z}_2$, then $G$ has no parallel edges, and
\item for every $e\in E(G)$, $(G-e,\beta)$ has a nowhere-zero flow.
\end{itemize}
\end{observation}

\begin{proof}
We may assume that $|\AA|\ge2$.  If $G$ were neither $2$-connected nor $K_2$, then
$G=G_1\cup G_2$ for proper induced connected subgraphs $G_1$ and $G_2$ intersecting
in exactly one vertex, and nowhere-zero flows in $\g/V(G_1)$ and $\g/V(G_2)$
would combine to a nowhere-zero flow in $\g$, which is a contradiction.

If $G$ had parallel edges between vertices $u$ and $v$, then (using the fact that $|\AA\setminus\{0\}|\ge 2$),
we could extend any nowhere-zero flow in $\g/\{u,v\}$ to a nowhere-zero flow in $\g$.

For an edge $e=uv$, a nowhere-zero flow in $\g/\{u,v\}$ corresponds to a flow $f$ in $\g$ which is non-zero everywhere except for $e$.
Since $\g$ does not have a nowhere-zero flow, the value of $f$ on $e$ must be zero, and thus $f$ is a nowhere-zero
flow in $(G-e,\beta)$.
\end{proof}

Moreover, we will need a standard observation on non-existence of nowhere-zero flows in subgraphs.

\begin{observation}\label{obs-restr}
Let $\AA$ be an abelian group and let $\g=(G,\beta)$ be a $\AA$-bordered graph with no nowhere-zero flow.
Let $B$ be a subset of\/ $V(G)$ such that $G[B]$ is connected and let $f$ be a nowhere-zero flow in $\g/B$, which we view as
assigning values to the edges of the symmetric orientation of $G-E(G[B])$.  Let $\beta_f$ be
defined by
$$\beta_f(u)=\beta(u)-\sum_{e=uv\in E(G)\setminus E(G[B])} f(e_v)$$
for every $u\in B$.  Then $\beta_f$ is a $\AA$-boundary for $G[B]$ and $(G[B],\beta_f)$ does not have a nowhere-zero flow.
\end{observation}

\begin{proof}
Let $b$ be the vertex of $G/B$ created by the contraction of $B$.  Since $G[B]$ is connected and
$$\sum_{u\in B}\beta_f(u)=(\beta/B)(b)-\sum_{e=bv\in E(G/B)} f(e_v)=0,$$
$\beta_f$ is a $\AA$-boundary for $G[B]$.
A nowhere-zero flow in $(G[B],\beta_f)$ would combine with $f$ to a
nowhere-zero flow in $\g$, and thus no such nowhere-zero flow exists.
\end{proof}

Let $e_1=u_1v$ and $e_2=u_2v$ be edges incident with the same vertex $v$.  By \emph{splitting off}
these edges, we mean deleting $e_1$ and $e_2$ and adding a new edge between $u_1$ and $u_2$.

\begin{observation}\label{obs-split}
Let $\AA$ be an abelian group, let $(G,\beta)$ be a $\AA$-bordered graph, and let $G'$ be
obtained from $G$ by splitting off a pair of edges.  If $(G',\beta)$ has a nowhere-zero flow, then so does $(G,\beta)$.
\end{observation}

\subsection{4-critical and $\zz$-flow-critical planar graphs}

A graph $H$ is an \emph{Ore sum} of $2$-connected graphs $H_1$ and $H_2$ if $H$ is obtained from the disjoint union of $H_1$ and $H_2$
by
\begin{itemize}
\item selecting a vertex $z$ of $H_1$,
\item splitting $z$ into two vertices $x_1$ and $y_1$, distributing the edges incident with $z$
arbitrarily among $x_1$ and $y_1$, but so that neither of $x_1,y_1$ is isolated in the resulting graph,
\item deleting an edge $e=x_2y_2$ from $H_2$, and
\item identifying $x_1$ with $x_2$ and $y_1$ with $y_2$.
\end{itemize}
Observe that $H$ is also $2$-connected.  If $H$ is a plane graph, then note that contracting the
subgraph corresponding to $H_2-e$ gives a plane drawing of $H_1$ in which the edges incident with $x_1$ (as well as the edges
incident with $y_1$) appear consecutively around $z$; and moreover, contracting the subgraph corresponding to $H_1-x_1$
and suppressing the arising parallel edges gives a plane drawing of $H_2$.  We say that $H$ is obtained by a \emph{plane Ore sum}
of the corresponding plane graphs $H_1$ and $H_2$.  A graph $H$ is \emph{$4$-Ore} if it can be obtained by Ore sums from copies of $K_4$.
\begin{observation}\label{obs-planeore}
If a $4$-Ore graph $H$ is planar, then any plane drawing of $H$ is obtained by plane Ore sums from plane drawings of $K_4$.
\end{observation}
Let us now introduce the dual operation.  A \emph{gluing} of plane graphs $G_1$ and $G_2$ is a plane graph $G$ obtained
from the disjoint union of $G_1$ and $G_2$ by
\begin{itemize}
\item selecting any distinct vertices $u_1$ and $v_1$ incident with the same face of $G_1$,
\item deleting an edge $e=u_2v_2$ from $G_2$, and
\item identifying $u_1$ with $u_2$ and $v_1$ with $v_2$.
\end{itemize}
\begin{observation}\label{obs-dual}
Let $H_1$, $H_2$, and $H$ be plane 2-connected graphs, and let $G_1$, $G_2$, and $G$ be their plane duals.
Then $H$ is a plane Ore sum of $H_1$ and $H_2$ if and only if $G$ is a gluing of $G_1$ and $G_2$.
\end{observation}
We say that a graph $G$ is \emph{dual $4$-Ore} if a plane drawing of $G$ is dual to a plane drawing of a $4$-Ore graph,
or equivalently (by Observation~\ref{obs-dual}), a plane drawing of $G$ is obtained by gluing from plane drawings of $K_4$.

Kostochka and Yancey~\cite{kostochka2018brooks} gave a result on the density of $k$-critical graphs, which can in the case $k=4$
be stated as follows.

\begin{theorem}[{Kostochka and Yancey~\cite[Theorem 6]{kostochka2018brooks}}]\label{thm-koyan}
If $H$ is a 4-critical graph, then
$$|E(H)|\ge \frac{5|V(H)|-2}{3}.$$
Moreover, if $H$ additionally is not a 4-Ore graph, then
$$|E(H)|\ge \frac{5|V(H)|-1}{3}.$$
\end{theorem}

Let us remark that every 4-Ore graph $H$ has exactly $(5|V(H)|-2)/3$ edges.
We say that a graph is \emph{exceptional} if it is $K_2$ or a dual $4$-Ore graph; note that every
exceptional graph $G$ satisfies $|E(G)| = (5|V(G)| - 8)/2$.
By Theorem~\ref{thm-tutteflow}, a plane graph has a $3$-coloring if and only if its dual has
a nowhere-zero $\zz$-flow.  Since deleting an edge in a plane graph corresponds to contracting the corresponding edge in the dual,
a plane graph is $4$-critical if and only if its dual is $\zz$-flow-critical.  Moreover, if $G$ is the dual of a plane graph $H$, then $|E(H)|=|E(G)|$ and by Euler's formula $|V(H)|=|E(G)|-|V(G)|+2$.
Hence, Theorem~\ref{thm-koyan} has the following consequence.

\begin{corollary}\label{cor-koyan}
If $G$ is a $\zz$-flow-critical planar graph, then
$$|E(G)|\le \frac{5|V(G)|-8}{2},$$
and if $G$ additionally is not exceptional, then
$$|E(G)|\le \frac{5|V(G)|-9}{2}.$$
\end{corollary}

Note that the special case of $K_2$ in Corollary~\ref{cor-koyan} is dual to the single-vertex graph with a loop,
which is 4-critical but not mentioned in Theorem~\ref{thm-koyan} which only considers loopless graphs.
Let us remark that $4$-Ore graphs are known to be $4$-critical,
and thus the following claim (implying that we cannot exclude any of the tight cases in Corollary~\ref{cor-koyan}) follows by duality.

\begin{observation}\label{obs-orecrit}
Every dual $4$-Ore graph is $\zz$-flow-critical.
\end{observation}
\begin{proof}
First, let us show the following auxiliary claim: If $H$ is a proper subgraph of a $4$-Ore graph $G$,
then
$$|E(H)|\le\frac{5|V(H)|-5}{3}.$$
We prove the claim by induction on the construction of $G$.  If $G=K_4$, then the claim follows by a straightforward
case analysis.  Hence, we can assume that $G$ is obtained as an Ore sum of smaller $4$-Ore graphs $G_1$ and $G_2$.
Let $z$, $x_1$, $x_2$, $y_1$, and $y_2$ be as in the definition of Ore sum.  Let $x$ and $y$ be the vertices of $G$ obtained
by the identification of $x_1$ with $x_2$ and $y_1$ with $y_2$.   Without loss of generality, we can assume $H$ is connected,
as otherwise the bound follows by considering each component of $H$ separately.
If $x,y\not\in V(H)$, then $H$ is a proper subgraph of $G_1$ or $G_2$ and the bound follows by the induction hypothesis.

Suppose now that $|V(H)\cap \{x,y\}|=1$, say $x\in V(H)$ but $y\not\in V(H)$.  For $i\in\{1,2\}$, let $H_i$ be the graph obtained from $H-(V(G_{3-i})\setminus \{x\})$
by renaming $x$ to $x_i$.  Then $H_i$ is a proper subgraph of $G_i$, and by the induction hypothesis, we have
$$|E(H)|=|E(H_1)|+|E(H_2)|\le \frac{5|V(H_1)|-5}{3}+\frac{5|V(H_2)|-5}{3}=\frac{5|V(H)|-5}{3}.$$
Finally, let us consider the case $x,y\in V(H)$.  Let $H_1$ be the graph obtained from $H-(V(G_2)\setminus \{x_2,y_2\})$ by identifying
$x$ with $y$ and renaming the resulting vertex to $z$, so that $H_1\subseteq G_1$.  Let $H_2$ be the graph obtained from $H-(V(G_1)\setminus \{x_1,y_1\})$
by renaming $x$ and $y$ to $x_2$ and $y_2$ and adding the edge $x_2y_2$, so that $H_2\subseteq G_2$.
Note that if $H_i=G_i$ for some $i\in\{1,2\}$, then $H_i$ is a $4$-Ore graph,
and thus $|E(H_i)|=\frac{5|V(H_i)|-2}{3}$.  However, since $H\neq G$, we have $H_1\neq G_1$ or $H_2\neq G_2$.
Therefore, by the induction hypothesis, we have
$$|E(H)|=|E(H_1)|+|E(H_2)|-1\le \frac{5|V(H_1)|-2}{3}+\frac{5|V(H_2)|-2}{3}-2=\frac{5|V(H)|-5}{3},$$
as required.

Consider any $4$-Ore graph $G$.  It is easy to prove by induction that $G$ is not 3-colorable.
We claim that $G$ is $4$-critical---otherwise, $G$ would have a proper $4$-critical subgraph $H$,
and we have shown that $|E(H)|\le \frac{5|V(H)|-5}{3}$, in contradiction to Theorem~\ref{thm-koyan}.
Since every $4$-Ore graph is $4$-critical, it follows that every dual $4$-Ore graph is $\zz$-flow-critical.
\end{proof}

We will also need the following property of exceptional graphs.

\begin{lemma}\label{lemma-single}
Let $G$ be an exceptional graph.
If $\beta$ is a non-zero $\zz$-boundary for $G$, then $(G,\beta)$ has a nowhere-zero flow.
\end{lemma}

\begin{proof}
We prove the claim by induction on the number of vertices of $G$.
It is easy to verify that the claim is true for $K_2$ and $K_4$.  Hence, suppose that $G$ is obtained
by gluing dual $4$-Ore graphs $G_1$ and $G_2$.
Let $u_1$, $v_1$, $u_2$, $v_2$, and $e=u_2v_2$ be as in the definition of gluing. Let $u$ and $v$ be the vertices of $G$ obtained by identifying $u_1,u_2$ and $v_1,v_2$, respectively.

Let $c_1=\sum_{x\in V(G_1)\setminus\{u_1,v_1\}} \beta(x)$.
Suppose first that there exists $x_0\in V(G_1)\setminus\{u_1,v_1\}$ such that $\beta(x_0)\neq 0$.
Let $\beta_2(y)=\beta(y)$ for $y\in V(G_2)\setminus \{u_2,v_2\}$, $\beta_2(v_2)=\beta(v)$,
and $\beta_2(u_2)=\beta(u)+c_1$.
Note that $(G_2,\beta_2)$ has a flow $f_2$ which is nowhere-zero except possibly for the edge $e$,
by the induction hypothesis if $\beta_2$ is non-zero, and by $\zz$-flow-criticality of $G_2$ and Observation~\ref{obs-basic} otherwise.
Let $\beta_1(x)=\beta(x)$ for $x\in V(G_1)\setminus\{u_1,v_1\}$
and $\beta_1(x)=\beta(x)-\sum_{h=xy\in E(G_2),h\neq e} f_2(h_y)$ for $x\in \{u_1,v_1\}$.
Note that $\beta_1$ is non-zero, since $\beta_1(x_0)=\beta(x_0)\neq 0$.
By the induction hypothesis, $(G_1,\beta_1)$ has a nowhere-zero flow, which combines with $f_2$ to a nowhere-zero flow for $(G,\beta)$.

Therefore, we can assume $\beta(x)=0$ for every $x\in V(G_1)\setminus\{u_1,v_1\}$, and in particular $c_1=0$
and the restriction $\beta'_2$ of $\beta$ to $V(G_2)$ must be non-zero.
By the induction hypothesis, there exists a nowhere-zero flow $f_2$ for $(G_2,\beta'_2)$.
By symmetry, we can assume that $f_2(e_{v_2})=1$.  Let $\beta'_1$ be the boundary function for $G_1$ which is zero everywhere
except for $u_1$ and $u_2$, with $\beta'_1(u_1)=1$ and $\beta'_1(v_1)=-1$.  By the induction hypothesis,
$(G_1,\beta'_1)$ has a nowhere-zero flow, which combines with $f_2$ to a nowhere-zero flow in $(G,\beta)$.
\end{proof}

Lemma \ref{lemma-single} has the following useful consequence.
\begin{corollary}\label{cor-zerobo}
Let $\g=(G,\beta)$ be a flow-critical $\zz$-bordered graph.  Let $B\subseteq V(G)$ be a set such that $G[B]$ is exceptional.
If $\beta(x)=0$ for every $x\in V(G)\setminus B$, then $\beta$ is the zero function.
\end{corollary}
\begin{proof}
Since $G[B]$ is exceptional, it is connected and has at least two vertices.  Hence $\g/B$ has a nowhere-zero flow $f$.
Let $\beta_f$ be the $\zz$-boundary for $G[B]$ defined in Observation~\ref{obs-restr}, so $(G[B],\beta_f)$ does not have
a nowhere-zero flow.  By Lemma~\ref{lemma-single}, it follows that $\beta_f$ is the zero function.
Since $\beta(x)=0$ for every $x\in V(G)\setminus B$, $-f$ (the flow assigning to every edge the value opposite to the
value assigned by $f$) is also a nowhere-zero flow in $\g/B$, and by the same argument, $\beta_{-f}$ is also the
zero function.  However, that means
$$0=\beta_f(v)+\beta_{-f}(v)=2\beta(v)=-\beta(v)$$
for every $v\in B$, and thus $\beta$ has zero values on $B$ as well.
\end{proof}

\subsection{Genus}

We need the following observation on Euler genus.

\begin{lemma}\label{lemma-conncon}
Let $G$ be a graph and $B$ a subset of its vertices.  If $G[B]$ is connected, then
$$g(G)\ge g(G/B)+g(G[B]).$$
\end{lemma}

\begin{proof}
Consider a drawing of $G$ on a surface $\Sigma$ of Euler genus $g(G)$, and the
induced drawing of $G-B$.  Since $G[B]$ is connected, $B$ is drawn within a single face $f$ of $G-B$.
Let $\Sigma'$ be the surface with boundary whose interior is homeomorphic to $f$; note that $G[B]$
can be drawn in $\Sigma'$, and thus $g(\Sigma')\ge g(G[B])$.  Let $\Sigma''$ be the surface obtained from $\Sigma-f$
by gluing it with the sphere $f'$ with the same number of holes as $\Sigma'$; we have
$g(\Sigma'')=g(\Sigma)-g(\Sigma')\le g(G)-g(G[B])$.  Moreover, $G/B$ can be drawn in $\Sigma''$, with the vertex
obtained by the identification of the vertices in $B$ drawn in $f'$, and thus
$g(G/B)\le g(\Sigma'')$.
\end{proof}

\section{The density of flow-critical graphs on surfaces}\label{sec-main}

Let us now proceed with the proof of Theorem~\ref{thm-main2}, which we now reformulate and strengthen.
For a graph $G$, let us define
$$\pi(G)=5|V(G)|-2|E(G)|+5g(G).$$
Note that if $G$ is exceptional, then $\pi(G)=8$.
We say that $G$ is \emph{sparse} if $G$ is exceptional or $\pi(G)\ge 9$.

\begin{theorem}\label{thm-main3}
Let $\g=(G,\beta)$ be a $\zz$-bordered graph.  If $\g$ is flow-critical,
then $G$ is sparse.
\end{theorem}

The rest of this section is devoted to the proof of Theorem \ref{thm-main3}.
A $\zz$-bordered graph is a \emph{minimal counterexample} if it is flow-critical, not sparse, and every $\zz$-bordered flow-critical graph with smaller number of vertices is sparse.
We will prove Theorem~\ref{thm-main3} by showing that there are no minimal counterexamples.
Note that by Observation~\ref{obs-basic}, if $(G,\beta)$ is a minimal counterexample, then $G$ is a $2$-connected simple graph.
Corollary~\ref{cor-koyan} can be restated as follows.
\begin{observation}\label{obs-koyan}
If $G$ is a $\zz$-flow-critical planar graph, then $G$ is sparse.
Hence, every minimal counterexample with zero boundary function is non-planar.
\end{observation}

Moreover, small graphs are not counterexamples.

\begin{observation}\label{obs-size}
If $\g=(G,\beta)$ is a minimal counterexample, then $|V(G)|\ge 5$.
\end{observation}
\begin{proof}
Since $\g$ is flow-critical, we have $|V(G)|\ge 2$.
If $|V(G)|=2$, then $G=K_2$ is exceptional.
If $|V(G)|=3$, then $\pi(G)\ge 5\cdot 3 - 2\cdot 3=9$, and $G$ is sparse.
If $|V(G)|=4$, then either $G=K_4$ is exceptional, or $|E(G)|\le 5$ and $\pi(G)\ge 5\cdot 4 - 2\cdot 5=10$.
\end{proof}

For a $G$-connected partition $\PP$ of the vertex set of a graph $G$,
let $k(G,\PP)$ be the number of parts of $\PP$ inducing an exceptional subgraph of $G$,
let $n(G,\PP)$ be the number of parts of size greater than one that do not induce an exceptional subgraph, and let
$$w(G,\PP)=4n(G,\PP)+3k(G,\PP).$$
In case the graph $G$ is clear from the context, we use $k(\PP)$, $n(\PP)$, and $w(\PP)$ for brevity.
Let us now state the crucial lemma forming the basis for the application of the potential method.

\begin{lemma}\label{lemma-add}
If $\g=(G,\beta)$ is a minimal counterexample, then for every $G$-connected partition $\PP$ of $V(G)$ with at least two parts, we have
$$\pi(G/\PP)\le \pi(G)-w(G,\PP).$$
\end{lemma}
\begin{proof}
We prove the claim by reverse induction on $|\PP|$.  If $|\PP|=|V(G)|$, i.e., $\PP$ is trivial, then $w(\PP)=0$
and the claim clearly holds.  Hence, suppose $|\PP|<|V(G)|$ and that the claim holds for every $G$-connected partition with more parts.

Hence, we can assume that $\PP$ contains a part $B$ of size at least two.
Let $f$ be a nowhere-zero flow in $\g/B$ and let $\beta_f$ be as in Observation~\ref{obs-restr},
so that $\g_B=(G[B],\beta_f)$ does not have a nowhere-zero flow.
Consequently, there exists a flow-critical contraction $\g_B/\QQ$
of $\g_B$. Let $\PP'=(\PP\setminus\{B\})\cup \QQ$.  Note that $|\QQ|\ge 2$, and thus $|\PP'|>|\PP|$.
Let us make the following three observations:
\begin{align}
  |V(G/\PP')| &= |V(G/\PP)| + |V(G[B]/\QQ)| - 1\label{eq:L3.4_1}\\
  |E(G/\PP')| &= |E(G/\PP)| + |E(G[B]/\QQ)|\label{eq:L3.4_2}\\
  \pi(G/\PP') &= \pi(G/\PP) + \pi(G[B]/\QQ) - 5\nonumber\\
  &\hphantom{=}- 5(g(G/\PP)+g(G[B]/\QQ)-g(G/\PP')).\label{eq:L3.4_2p}
\end{align}
From these observations\footnote{In the rest of the paper we shall frequently use analogues of (\ref{eq:L3.4_1})--(\ref{eq:L3.4_2p}) without going through the details.},
using Lemma \ref{lemma-conncon} and the induction hypothesis for $\PP'$, we obtain the following:
\begin{align}
\pi(G/\PP)&=\pi(G/\PP')-\pi(G[B]/\QQ)+5+5(g(G/\PP)+g(G[B]/\QQ)-g(G/\PP')) \nonumber\\
&\le \pi(G/\PP')-\pi(G[B]/\QQ)+5 \nonumber\\
&\le \pi(G)-w(\PP')-\pi(G[B]/\QQ)+5. \label{eq:L3.4_3}
\end{align}

If $G[B]$ is exceptional, then $G[B]$ is $\zz$-flow-critical by Observation~\ref{obs-orecrit}
and $\beta_f$ is the zero function by Lemma~\ref{lemma-single}, and thus $\QQ$ is a trivial partition;
this implies that $n(\PP')=n(\PP)$, $k(\PP')=k(\PP)-1$, and $\pi(G[B]/\QQ)=\pi(G[B])=8$.  It follows that
$$\pi(G/\PP)\le \pi(G)-w(\PP)+3-8+5=\pi(G)-w(\PP).$$
If $G[B]$ is not exceptional, but $G[B]/\QQ$ is, then $\QQ$ cannot be a trivial partition,
and thus $n(\PP')\ge n(\PP)-1$, $k(\PP')\ge k(\PP)$, and $n(\PP')+k(\PP')\ge n(\PP)+k(\PP)$.
This implies that $w(\PP') \ge w(\PP) - 1$.  Since $G[B]/\QQ$ is exceptional, we have $\pi(G[B]/\QQ)=8$, and using (\ref{eq:L3.4_3}), we obtain:
\begin{align*}
  \pi(G/\PP) &\le \pi(G) - w(\PP') - \pi(G[B]/Q) + 5 \\
             &\le \pi(G)-w(\PP)+1-8+5<\pi(G)-w(\PP).
\end{align*}

Finally, if neither $G[B]$ nor $G[B]/\QQ$ is exceptional, then $n(\PP')\ge n(\PP)-1$ and $k(\PP')\ge k(\PP)$.
Moreover, since $\PP$ has at least two parts, $G[B]$ (and thus also $G[B]/\QQ$) has fewer vertices than $G$,
and since $\g$ is a minimal counterexample, $G[B]/\QQ$ is sparse, i.e., $\pi(G[B]/\QQ)\ge 9$.
Therefore,
$$\pi(G/\PP)\le \pi(G)-w(\PP)+4-9+5=\pi(G)-w(\PP).$$
This completes the proof.
\end{proof}

Since a minimal counterexample $\g=(G,\beta)$ satisfies $\pi(G)<9$, we have the
following corollary.

\begin{corollary}\label{cor-add}
If $\g=(G,\beta)$ is a minimal counterexample, then for every $G$-connected partition $\PP$ of $V(G)$ with at least two parts, we have
$$\pi(G/\PP)<9-w(G,\PP).$$
\end{corollary}

As an application, we can restrict edge-connectivity of counterexamples.

\begin{corollary}\label{cor-connect}
If $\g=(G,\beta)$ is a minimal counterexample, then $G$ is $3$-edge-connected.
Moreover, the only 3-edge-cuts in $G$ are the neighborhoods of vertices of degree three,
and no two vertices of degree three are adjacent.
\end{corollary}

\begin{proof}
Consider any minimal edge-cut $S$ in $G$, and let $X$ and $Y$ be the sides of $S$
with $|X|\le |Y|$.  Let $\PP=\{X,Y\}$.  Since $S$ is minimal, this partition is $G$-connected.
By Corollary~\ref{cor-add},
$$10-2|S|=\pi(G/\PP)<9-w(\PP),$$
and thus 
\begin{equation}
  |S|>\frac{w(\PP)+1}{2}. \label{eq:C36}
\end{equation}
By Observation~\ref{obs-size}, $G$ has at least five vertices, and thus $|Y|\ge \lceil |V(G)|/2\rceil\ge 3$.
Therefore, $n(\PP)+k(\PP)\ge 1$ and $w(\PP)\ge 3$, implying that $|S|\ge 3$.
Moreover, if $|X|\ge 2$, then $n(\PP)+k(\PP)\ge 2$ and $w(\PP)\ge 6$,
which implies $|S|\ge 4$.  Therefore, $G$ is $3$-edge-connected and the only
3-edge-cuts in $G$ are the neighborhoods of vertices of degree three.

Suppose now that $X$ consists of two adjacent vertices $x_1$ and $x_2$ of degree three.  We have $|S|=4$, and by (\ref{eq:C36}), we conclude that
$w(\PP)=6$, and thus $G[Y]$ is exceptional.  If $\beta(x_1)=\beta(x_2)=0$, then by Corollary~\ref{cor-zerobo}, $\beta$ is the zero function.
Moreover, we have
$$\pi(G)=\pi(G[X])+\pi(G[Y])-2|S|+5g(G)=8+5g(G).$$
Since $\pi(G)<9$, it follows that $g(G)=0$, i.e., $G$ is planar. However, this contradicts Observation~\ref{obs-koyan}.

Therefore, we can assume $\beta(x_1)=1$.  Let $e=x_1x_2$, let $e^1$ and $e^2$ be the edges incident with $x_1$
distinct from $e$, and let $h^1$ and $h^2$ be the edges incident with $x_2$ distinct from $e$.
Note that $\g/Y$ has nowhere-zero flows $f_1$ and $f_2$ such that
$f_1(e_{x_2})=f_2(e_{x_2})=1$, $f_1(h^j_{x_2})=f_2(h^j_{x_2})$ for $j\in \{1,2\}$,
$f_1(e^1_{x_1})=1$, $f_2(e^1_{x_1})=-1$, $f_1(e^2_{x_1})=-1$ and $f_2(e^2_{x_1})=1$.
Let $\beta_{f_1}$ and $\beta_{f_2}$ be as in Observation~\ref{obs-restr}, so that
neither $(G[Y],\beta_{f_1})$ nor $(G[Y],\beta_{f_2})$ has a nowhere-zero flow.
Since $f_1$ and $f_2$ differ exactly on the edges $e^1$ and $e^2$ and $G$ does not have
parallel edges, the functions $\beta_{f_1}$ and $\beta_{f_2}$ are different, and thus
at least one of them is non-zero.  However, this contradicts Lemma~\ref{lemma-single}.
\end{proof}

Next, let us consider vertices of degree three.

\begin{lemma}\label{lemma-nottri}
Let $v$ be a vertex of degree three in a minimal counterexample $\g=(G,\beta)$.
If $\beta(v)=0$, then $v$ is contained in at most one triangle.
\end{lemma}

\begin{proof}
Let $v_1$, $v_2$, and $v_3$ be the neighbors of $v$, joined to $v$
by edges $e^1$, $e^2$, and $e^3$.  Suppose for a contradiction that $v_1v_2,v_2v_3\in E(G)$,
and let these edges be denoted by $e^{12}$ and $e^{32}$.
Let $\g'=(G',\beta')$ be the $\zz$-bordered graph obtained from $\g/\{v_1,v,v_3\}$ by
removing two edges of the resulting triple edge between $v_2$ and the vertex $z$ arising from the contraction.
Let $e'$ be the remaining edge between $v_2$ and $z$.

We claim that $\g'$ does not have a nowhere-zero flow.  Indeed, suppose for a contradiction that $f$ is a nowhere-zero flow in $\g'$.
We can view $f$ as assigning values to the edges of the symmetric orientation of $G_0=G-\{v_1v_2,v_2v_3,vv_1,vv_2,vv_3\}$.
For $i\in \{1,2,3\}$, let $\delta_i=\beta(v_i)-\sum_{e=v_iy\in E(G_0)} f(e_y)$.  Note that
$\delta_2=f(e'_z)\neq 0$, and by symmetry, we can assume that $\delta_2=1$.  Moreover, since $\beta(v)=0$, we have
$\delta_1+\delta_3=f(e'_{v_2})=-f(e'_z)=-\delta_2=-1$.  In particular, at most one of $\delta_1$ and $\delta_3$ is equal to $0$,
and by symmetry, we can assume $\delta_1\neq 0$.  We can extend $f$ to $G$ as follows.
\begin{itemize}
\item If $\delta_1=1$ (and thus $\delta_3=1$), then set
$f(e^i_{v_i})=f(e^{i2}_{v_i})=1$ for $i\in \{1,3\}$ and $f(e^2_v)=-1$.
\item If $\delta_1=-1$ (and thus $\delta_3=0$), then set
$f(e^1_{v_1})=f(e^{12}_{v_1})=-1$ and $f(e^2_{v})=f(e^{32}_{v_3})=f(e^3_v)=1$.
\end{itemize}
In either case, we obtain a nowhere-zero flow in $\g$, which is a contradiction.

Since $\g'$ does not have a nowhere-zero flow, there exists a $G'$-connected partition $\PP'$ of $V(G')$ such that $\g'/\PP'$
is flow-critical.  Note that $v_2$ and $z$ belong to different parts $A$ and $B$ of $G'/\PP'$,
as if they both belonged to the same part $C$, then we would have $G'/\PP'=G/\PP_0$
for the $G$-connected partition $\PP_0$ obtained by replacing $C$ by $(C\setminus \{z\})\cup\{v_1,v,v_3\}$,
contradicting the flow-criticality of $G$.  Let $B_0=(B\setminus \{z\})\cup\{v_1,v_3\}$.
If $G[B_0]$ is connected, then let $\PP=(\PP'\setminus \{B\})\cup \{\{v\},B_0\}$.
Otherwise, note that $G[B_0]$ has exactly two components $B_1$ and $B_3$,
containing $v_1$ and $v_3$, respectively, and let $\PP=(\PP'\setminus \{B\})\cup \{\{v\},B_1,B_3\}$.
Note that in either case, $\PP$ is a $G$-connected partition of $V(G)$,
and moreover, $G'/\PP'$ is a minor of $G/\PP$, and thus $g(G'/\PP')\le g(G/\PP)$.
Applying Corollary~\ref{cor-add}, in the former case we obtain
\begin{align*}
9-w(G,\PP)&>\pi(G/\PP)=\pi(G'/\PP')+5-4\cdot 2+5(g(G/\PP)-g(G'/\PP'))\\
&\ge \pi(G'/\PP')-3,
\end{align*}
and in the latter case,
\begin{align*}
9-w(G,\PP)&>\pi(G/\PP)=\pi(G'/\PP')+2\cdot 5-4\cdot 2+5(g(G/\PP)-g(G'/\PP'))\\
&\ge \pi(G'/\PP')+2.
\end{align*}
The latter is not possible, since $w(G,\PP)\ge 0$ and $\pi(G'/\PP')\ge 8$.
In the former case, note that $|B_0|\ge 2$, and thus $n(G,\PP)+k(G,\PP)\ge 1$.
Hence, the inequality can only hold if $n(G,\PP)=0$, $k(G,\PP)=1$,
and $\pi(G'/\PP')=8$; that is, $B_0$ is the only non-trivial part of $\PP$,
$G[B_0]$ is exceptional, and $G'/\PP'=G'/B$ is exceptional.  By Lemma~\ref{lemma-single},
the boundary function of $\g'/B$ is zero, and thus $\beta(x)=0$ for every
$x\in V(G')\setminus B=V(G)\setminus (B_0\cup \{v\})$.  Since $\beta(v)=0$ by the assumption
and since $G[B_0]$ is exceptional, Corollary~\ref{cor-zerobo} implies that $\beta$ is the zero function.

Since both $G[B_0]$ and $G'/B$ are exceptional, they are both planar.  Consequently,
\begin{align*}
  \pi(G) &= 5|V(G)|-2|E(G)|+5g(G)\\
  &= (5|V(G'/B)|-2|E(G'/B))-4\cdot 2+(5|B_0|-2|E(G[B_0])|)+5g(G)\\
  &= \pi(G'/B)-8+\pi(G[B_0])+5g(G) \\
  &= 8+5g(G).
\end{align*}
Since $\pi(G)<9$, this implies that $G$ is planar.  However, this contradicts Observation~\ref{obs-koyan}.
\end{proof}

The above result is useful in conjunction with the following observation.

\begin{lemma}\label{lemma-tri3}
If $G$ is a dual $4$-Ore graph and $G\neq K_4$, then there exists a set $X\subseteq V(G)$ of size four such that
\begin{itemize}
\item[\rm (i)] every vertex in $X$ has degree three in $G$ and is contained in at least two triangles,
\item[\rm (ii)] $G[X]$ has maximum degree at most one,
\item[\rm (iii)] if three vertices of $X$ form an independent set in $G$, then they do not have a common neighbor, and
\item[\rm (iv)] if two adjacent vertices of $X$ have a common neighbor, then they have two common neighbors.
\end{itemize}
\end{lemma}

\begin{proof}
We prove the claim by induction on $|V(G)|$.
The claim is trivially true if $G=K_4$, and thus we can assume this is not the case.
Hence, $G$ is obtained from dual $4$-Ore graphs $G_1$ and $G_2$ by gluing. 
Let $u_1$, $v_1$, $u_2$, $v_2$, and $e=u_2v_2$ be as in the definition of gluing.
Let $u$ and $v$ be the vertices of $G$ obtained by identifying $u_1,u_2$ and $v_1,v_2$, respectively.
For $i\in \{1,2\}$, let $X_i\subseteq V(G_i)$
be obtained by the induction hypothesis if $G_i\neq K_4$ and $X_i=V(G_i)$ if $G_i=K_4$.
Let $Y_i$ be a subset of $X'_i=X_i\setminus \{u_i,v_i\}$ of size two obtained as follows.
If $G_i[X'_i]$ has an edge whose ends have a common neighbor, then let $Y_i$ consists of the ends of such an edge.
In case that $G_i[X'_i]$ has no such edge, then note that (iii) implies that no three vertices of $X'_i$ have
a common neighbor.
\begin{itemize}
\item If $|X'_i|=2$, then let $Y_i=X'_i$.
\item If $|X'_i|=3$, we can assume that $v_i\in X_i$ (the case $u_i\in X_i$ is symmetric).
There exists a vertex $x_i\in X'_i$ non-adjacent to $u_i$, since no three vertices of $X'_i$ have
a common neighbor.  Moreover, by (ii), there exists a vertex $y_i\in X'_i\setminus\{x_i\}$
non-adjacent to $v_i$.  Let $Y_i=\{x_i,y_i\}$.
\item If $|X'_i|=4$, then since no three vertices of $X'_i$ have a common neighbor, there exists a vertex $x_i\in X'_i$ non-adjacent
to $u_i$, and a vertex $y_i\in X'_i\setminus\{x_i\}$ non-adjacent to $v_i$.  Let $Y_i=\{x_i,y_i\}$.
\end{itemize}
Note that either
\begin{itemize}
\item[(a)] the two vertices in $Y_i$ are adjacent and have two common neighbors, or
\item[(b)] the two vertices in $Y_i$ can be labeled as $x_i$ and $y_i$ so that $x_iu_i, y_iv_i\not\in E(G_i)$.
\end{itemize}
Indeed, (a) occurs when $G_i[X'_i]$ has an edge whose ends have a common neighbor by (iv) when $G_i\neq K_4$
and trivially when $G_i=K_4$.  Otherwise, (b) occurs; this is explicit in the last two cases, and follows by (ii) in the first case.

We claim that $X=Y_1\cup Y_2$ satisfies (i)--(iv).  Clearly, every vertex of $X$ has degree three in $G$, and the vertices of $Y_1$
are contained in at least two triangles.  If a vertex $z\in Y_2$ is contained in a triangle in $G_2$ that does not appear in
$G$, then $z$ is adjacent both to $u_2$ and to $v_2$, and thus $Y_2$ does not satisfy (b).  Hence, it satisfies (a),
the vertices of $Y_2$ are adjacent and have two common neighbors, and thus they are contained in at least two triangles.
Hence, $X$ satisfies the condition (i).

There are no edges in $G$ between $Y_1$ and $Y_2$, implying that (ii) holds.  Consider any three vertices
in $X$ forming an independent set; the two vertices of $Y_i$ for some $i\in\{1,2\}$, and a vertex from $Y_{3-i}$.
Note that their common neighbors can only be $u$ or $v$.  Since the vertices of $Y_i$ are non-adjacent, $Y_i$
does not satisfy (a), and thus it satisfies (b).  Since $x_i$ is not adjacent to $u_i$ and $y_i$ is not adjacent to $v_i$,
the three vertices do not have a common neighbor.  Therefore, (iii) holds.  Finally, if two vertices of $X$ are adjacent and have a common neighbor,
then for some $i\in\{1,2\}$, they both belong to $Y_i$, and they have two common neighbors in $G_i$ by (iv) if $G_i\neq K_4$
or trivially if $G_i=K_4$.  Hence, they also have two common neighbors in $G$, and (iv) holds.
\end{proof}

As the next step, let us reduce exceptional subgraphs in minimal counterexamples.
\begin{lemma}\label{lemma-noexc}
Let $\g=(G,\beta)$ be a minimal counterexample and let $A$ be a subset of its vertices.
If $|A|\ge 3$, then $G[A]$ is not exceptional.
\end{lemma}
\begin{proof}
Suppose for a contradiction that $G[A]$ is exceptional.  It is straightforward to prove by induction that
every exceptional graph other than $K_2$ contains $K_4$
as a subgraph, and thus by choosing $A$ minimal, we can assume that $G[A]=K_4$.
Let $\{A_1,A_2\}$ be a partition of $A$ into parts of size two.
Let $\g'=(G',\beta')$ be the $\zz$-bordered graph obtained from $\g/A_1/A_2$
by removing all but one edge $h$ between the vertices $a_1$ and $a_2$ resulting from the contraction of $A_1$ and $A_2$,
respectively.

Suppose that $\g'$ has a nowhere-zero flow $f'$.  We can view this flow as assigning non-zero values to all edges
of the symmetric orientation of $E(G)\setminus E(G[A])$.  Let $\g_A=(G[A],\beta_A)$, 
where $\beta_A(x)=\beta(x)-\sum_{e=xy\in E(G)\setminus E(G[A])} f'(e_y)$ for every $x\in A$.  Note that
$$\sum_{x\in A_1} \beta_A(x)=\beta'(a_1)-\sum_{e=a_1y\in E(G')\atop e\neq h} f'(e_y)=f'(h_{a_2})\neq 0,$$
and thus the boundary function of $\g_A$ is non-zero.  By Lemma~\ref{lemma-single}, $\g_A$ has a nowhere-zero flow;
however, this flow would combine with $f'$ to a nowhere-zero flow in $G$, which is a contradiction.

Therefore, $\g'$ does not admit a nowhere-zero flow, and thus there exists a $G'$-connected partition $\PP'$
of $V(G')$ such that $\g'/\PP'$ is flow-critical.  Note that $a_1$ and $a_2$ belong to different parts $B'_1$ and $B'_2$ of $\PP'$,
as otherwise $\g'/\PP'$ would be a proper contraction of $\g$, in contradiction to the flow-criticality of $\g$.
For $i\in\{1,2\}$, let $B_i=(B'_i\setminus\{a_i\})\cup A_i$, and let
$\PP=(\PP'\setminus \{B'_1,B'_2\})\cup \{B_1,B_2\}$. Clearly, $\PP$ is a $G$-connected partition of $V(G)$.
Note that $g(G/\PP)=g(G'/\PP')$, since the two graphs differ only in that $h$ is replaced by a quadruple edge.
By Corollary~\ref{cor-add}, we see that $9-w(G,\PP)>\pi(G/\PP)=\pi(G'/\PP')-6$.  Since $|B_1|,|B_2|\ge 2$, we
have $n(G,\PP)+k(G,\PP)\ge 2$ and $w(G,\PP)\ge 6$.  Hence, this is only possible if $n(\PP)=0$, $k(\PP)=2$, and $\pi(G'/\PP')=8$.
That is, $G[B_1]$, $G[B_2]$, and $G'/\PP'$ are exceptional graphs and $B_1$ and $B_2$ are the only non-trivial parts of $\PP$.
By Lemma~\ref{lemma-single}, $\beta'/\PP'$ is the zero function, and thus $\beta(x)=0$ for every $x\in V(G)\setminus (B_1\cup B_2)$.
Note that $G[B_1]\cup G[A]\cup G[B_2]$ is an exceptional graph, as it is obtained as follows:
\begin{itemize}
\item Let $A_1=\{u_2,v_2\}$.  If $B_1=A_1$, then let $H=G[A]$.  Otherwise, let $H_1$ be a copy of
the dual 4-Ore graph $G[B_1]$ with $u_2$ renamed to $u_1$ and $v_2$ renamed to $v_1$;
and let $H$ be the graph obtained by gluing $H_1$ (on $u_1$ and $v_1$) with $G[A]$ (on the edge $u_2v_2$).
In either case, $H$ is isomorphic to $G[B_1]\cup G[A]$, and it is a dual $4$-Ore graph.
\item Similarly, we then glue $G[B_2]$ with $H$ to obtain a graph isomorphic to $G[B_1]\cup G[A]\cup G[B_2]$.
\end{itemize}
The graph $G'$ does not have any edge parallel to $h$ by Observation~\ref{obs-basic}, and thus
$G-E(G[A])$ does not have any edges between $B_1$ and $B_2$; it follows that $G[B_1\cup B_2]=G[B_1]\cup G[A]\cup G[B_2]$.
As we have observed before, $\beta(x)=0$ for every $x\in V(G)\setminus (B_1\cup B_2)$, and thus Corollary~\ref{cor-zerobo} implies that $\beta$ is the zero function.  Since $G'/\PP'$, $G[B_1]$ and $G[B_2]$ are planar, we have
$$\pi(G)=\pi(G'/\PP')-2\cdot 5-3\cdot 2+\pi(G[B_1])+\pi(G[B_2])+5g(G)=8+5g(G).$$
Since $\pi(G)<9$, this implies that $G$ is planar.  However, this contradicts Observation~\ref{obs-koyan}.
\end{proof}

Let us now restrict triangles in minimal counterexamples.

\begin{lemma}\label{lemma-splitoff}
Let $\g=(G,\beta)$ be a minimal counterexample.  Let $e_1=u_1v$ and $e_2=u_2v$ be edges of $G$ incident with the same vertex $v$ of degree
at least four.  If $u_1u_2\in E(G)$, then $v$ is the only common neighbor of $u_1$ and $u_2$.
\end{lemma}

\begin{proof}
Suppose for a contradiction that $u_1$ and $u_2$ have another common neighbor $v'\neq v$.
Let $G'$ be the graph obtained from $G$ by splitting off $e_1$ and $e_2$, and let $e$ be the resulting edge between $u_1$ and $u_2$.
Since $G$ is $3$-edge-connected, $G'$ is connected.
Since $\g$ is flow-critical, Observation~\ref{obs-split} implies that $\g'=(G',\beta)$ does not have a nowhere-zero flow, and thus there exists
a $G'$-connected partition $\PP$ of $V(G)$ such that $\g'/\PP$ is flow-critical.  Since $u_1$ and $u_2$ are joined by a double edge in $G'$,
Observation~\ref{obs-basic} implies that $u_1$ and $u_2$ are in the same part $B$ of $\PP$.  If $v'$ were not in this part, then $G'/\PP$
would contain a double edge between the vertex $b$ corresponding to $B$ and the vertex corresponding to the part containing $v'$; by Observation~\ref{obs-basic}
this does not happen, and thus we conclude that $v'\in B$.  Note that $\PP$ is also $G$-connected, since $E(G')\setminus E(G)=\{e\}$ and $u_1$ and $u_2$
are adjacent in $G$.  Since $G'/\PP$ is not a proper contraction of $G$, we have $v\not\in B$.
Note also that $G'/\PP$ is a subgraph of $G/\PP$, and thus $g(G'/\PP)\le g(G/\PP)$.
By Corollary~\ref{cor-add}, we have
\begin{align*}
  9-w(G,\PP) & > \pi(G/\PP) \\
     & =\pi(G'/\PP)-4+5(g(G/\PP)-g(G'/\PP))\\
     &\ge \pi(G'/\PP)-4.
\end{align*}
Since $|B|\ge 3$, Lemma~\ref{lemma-noexc} implies that $G[B]$ is not exceptional, and thus $n(G,\PP)\ge 1$.
Therefore, $n(G,\PP)=1$, $B$ is the only non-trivial part of $\PP$, and $\pi(G'/\PP)=8$, i.e., $G'/\PP$ is exceptional.
By Lemma~\ref{lemma-single}, we have $\beta(x)=0$ for every $x\in V(G)\setminus B$.  Since $v$ has degree at least four,
Corollary~\ref{cor-connect} implies that $G'$ (and thus also $G'/\PP$) is 2-edge-connected, and thus $G'/\PP\neq K_2$.
If $G'/\PP=K_4$, then the two vertices of $G$ not belonging to $B\cup\{v\}$ have degree three and are adjacent,
contradicting Corollary~\ref{cor-connect}.

Therefore, $G'/\PP\not\in \{K_2,K_4\}$. Let $X$ be a set of four vertices of $G'/\PP$ of degree three, each contained in at least two triangles,
obtained using Lemma~\ref{lemma-tri3}. Let $x_1$ and $x_2$ be two vertices of $X$ different from $v$ and the vertex $b$
obtained by the contraction of $B$.  If $b\in X$, then by Lemma~\ref{lemma-tri3}(ii), we can assume that $x_1b\not\in E(G'/\PP)$.
If $b\not\in X$, then $X$ contains another vertex $x_3$ distinct from $x_1$, $x_2$, and $v$.  
For $i\in \{1,2,3\}$, the vertex $x_i$ has degree three in $G$ and $\beta(x_i)=0$.
Hence, Corollary~\ref{cor-connect} implies that $\{x_1,x_2,x_3\}$ is an independent set, and thus by Lemma~\ref{lemma-tri3}(iii), we can again assume that $x_1b\not\in E(G'/\PP)$.
But then $x_1$ is a vertex of degree three contained in at least two triangles in $G$ and satisfying $\beta(x_1)=0$, contradicting Lemma~\ref{lemma-nottri}.
\end{proof}

This has the following consequence on adjacency among triangles.

\begin{corollary}\label{cor-notria}
If $\g=(G,\beta)$ is a minimal counterexample, then
for every triangle $T$ in $G$, at most one edge of $T$ belongs to another
triangle.
\end{corollary}
\begin{proof}
By Lemma~\ref{lemma-splitoff}, if two of the edges of $T$ were contained in triangles different from $T$,
then two vertices of $T$ would have degree three, contradicting Corollary~\ref{cor-connect}.
\end{proof}

We are now ready to finish the proof.

\begin{proof}[Proof of Theorem~\ref{thm-main3}]
If Theorem~\ref{thm-main3} were false, then there would exist a minimal counterexample $\g=(G,\beta)$.
Consider a drawing of $G$ in a surface of Euler genus $g(G)$.  Since $G$ is simple (by Observation~\ref{obs-basic})
and has more than two vertices (by Observation~\ref{obs-size}), each face of $G$ has length at least three,
and each face of length three is bounded by a triangle.  Let $f_k$ denote the number
of faces of $G$ of length $k$ and let $f$ denote the number of all faces of $G$.
By Corollary~\ref{cor-notria}, every $3$-face of $G$ shares an edge with at least two faces of length at least four.
This implies that $f_3 \le \tfrac{1}{2}\sum_{k\ge4} kf_k$. Using this and the fact that $k-\tfrac{10}{3}\ge \tfrac{1}{6}k$ for every $k\ge 4$, we obtain the following inequality:
\begin{align*}
    2|E(G)| &= \sum_{k\ge 3} kf_k \\
      &= \tfrac{10}{3}f + \sum_{k\ge 3} (k-\tfrac{10}{3})f_k \\
      &\ge \tfrac{10}{3}f - \tfrac{1}{3}f_3 + \tfrac{1}{6}\sum_{k\ge 4}kf_k \\
      &\ge \tfrac{10}{3}f.
\end{align*}

By Euler's formula,
$$|E(G)|=|V(G)|+f+g(G)-2\le |V(G)|+\tfrac{3}{5}|E(G)|+g(G)-2.$$
Therefore,
$$10\le 5|V(G)|-2|E(G)|+5g(G)=\pi(G),$$
contradicting the assumption that $\g$ is a minimum counterexample.
\end{proof}

\section{Minimal counterexamples to the density conjecture}\label{sec-cex}

For a graph $H$, let $\sigma(H)=3|V(H)|-|E(H)|$.
We say a $\zz$-bordered graph $\g=(G,\beta)$ is a \emph{smallest counterexample} if $\g$ is flow-critical,
$\sigma(G)<5$, and every flow-critical $\zz$-bordered graph $(G',\beta')$ with less than $|V(G)|$ vertices
satisfies $\sigma(G')\ge 5$.

For a partition $\PP$, let $l(\PP)$ denote the number of parts of $\PP$ of size at least two.
The following claim is proved analogously to Corollary~\ref{cor-add}.

\begin{lemma}\label{lemma-addprime}
If $\g=(G,\beta)$ is a smallest counterexample, then for every partition $\PP$ of $V(G)$ with at least two parts, we have
$$\sigma(G/\PP)<5-2l(\PP).$$
\end{lemma}

\begin{proof}
We prove the claim by reverse induction on $|\PP|$.  If $|\PP|=|V(G)|$, i.e., $\PP$ is trivial, then $l(\PP)=0$
and the claim holds since $\g$ is a smallest counterexample.
Hence, suppose $|\PP|<|V(G)|$ and that the claim holds for every partition with more parts.

If $\PP$ is not $G$-connected, then there exists $B\in \PP$ and a partition $\{B_1,B_2\}$ of $B$ to non-empty parts
such that $G$ contains no edge between $B_1$ and $B_2$.  Let $\PP'=(\PP\setminus \{B\})\cup\{B_1,B_2\}$.
Then
$$\sigma(G/\PP)=\sigma(G/\PP')-3<2-2l(\PP')\le 4-2l(\PP)$$
by the induction hypothesis.

Hence, we can assume that $\PP$ is $G$-connected and contains a part $B$ of size at least two.
Let $f$ be a nowhere-zero flow in $\g/B$ and let $\beta_f$ be as in Observation~\ref{obs-restr},
so that $\g_B=(G[B],\beta_f)$ does not have a nowhere-zero flow.
Consequently, there exists a flow-critical contraction $\g_B/\QQ$
of $\g_B$. Let $\PP'=(\PP\setminus\{B\})\cup \QQ$.  Note that $|\QQ|\ge 2$, and thus $|\PP'|>|\PP|$.
Since $\g$ is a smallest counterexample and $|B|<|V(G)|$, we have $\sigma(G[B]/\QQ)\ge 5$.
By the induction hypothesis, it follows that
\begin{align*}
\sigma(G/\PP)&=\sigma(G/\PP')-\sigma(G[B]/\QQ)+3\\
&< 8-2l(\PP')-\sigma(G[B]/\QQ)\\
&\le 3-2l(\PP')\le 5-2l(\PP).
\end{align*}
This completes the proof.
\end{proof}

We can now establish the desired properties of smallest counterexamples.
\begin{proof}[Proof of Theorem~\ref{thm-smallcex}]
Since $\sigma(G)<5$ and $\sigma(K_2)=5$, we have $|V(G)|\ge 3$.
Consider any minimal edge-cut $S$ in $G$, and let $X$ and $Y$ be the sides of $S$ with $|X|\le |Y|$.
By Lemma~\ref{lemma-addprime},
$$6-|S|=\sigma(G/\PP)<5-2l(\PP),$$
and thus
$$|S|>1+2l(\PP).$$
Hence, if $|X|\ge 2$, we have $|S|>5$.
In particular, the only $(\le\!5)$-edge-cuts in $G$ are neighborhoods of vertices, and each vertex
has degree at least four.

Suppose now that $v$ is a vertex of $G$ of degree four, and let $v_1$ and $v_2$ be distinct neighbors of $v$.
Let $G'$ be the graph obtained from $G$ by splitting off the edges $v_1v$ and $v_2v$; by Observation~\ref{obs-split},
$(G',\beta)$ has no nowhere-zero flow.  Let $B=V(G)\setminus\{v\}$ and let $f$ be a nowhere-zero flow
in $(G',\beta)/B$ (which exists, since $G'/B$ consists of two vertices joined by an edge of multiplicity two).
Note that $G'[B]=G'-v$ is connected, since $G$ is $4$-edge-connected and $\deg v=4$.
Let $\beta_f$ be as in Observation~\ref{obs-restr}, so that $(G'-v,\beta_f)$ does not admit a nowhere-zero flow.
Consequently, there exists a partition $\PP_0$ of $B$ such that $(G'-v,\beta_f)/\PP_0$ is flow-critical,
and since $G$ is a smallest counterexample, we have $\sigma((G'-v)/\PP_0)\ge 5$.  Let $\PP_1=\PP_0\cup \{\{v\}\}$.
Let $\delta=1$ if $v_1$ and $v_2$ are in different parts of $\PP_0$ and $\delta=0$ otherwise.
By Lemma~\ref{lemma-addprime},
\begin{align*}
5-2l(\PP_1)&>\sigma(G/\PP_1)=\sigma((G'-v)/\PP_0)+\delta-1\\
&\ge 4+\delta.
\end{align*}
This is a contradiction, since if $0=l(\PP_1)=l(\PP_0)$, then $\delta=1$.
Therefore, $G$ has minimum degree at least five, and thus it is $5$-edge-connected.

Finally, suppose that $G$ contains a triangle.  Let $v_1$ and $v_2$ be adjacent vertices with a common neighbor $v$.
Let $G_2$ be the graph obtained from $G$ by splitting off the edges $v_1v$ and $v_2v$.
By Observation~\ref{obs-split}, $(G_2,\beta)$ has no nowhere-zero flows, and thus there exists a partition $\PP_2$ of $V(G_2)$
such that $(G_2,\beta)/\PP_2$ is flow-critical.  Note that $v_1$ and $v_2$ are joined by a double edge in $G_2$,
and by Observation~\ref{obs-basic}, they belong to the same part of $\PP_2$.  Consequently $l(\PP_2)\ge 1$
and $|V(G_2/\PP_2)|<|V(G)|$. Since $G$ is a smallest counterexample, we have $\sigma(G_2/\PP_2)\ge 5$.
By Lemma~\ref{lemma-addprime},
$$5-2l(\PP_2)>\sigma(G/\PP_2)=\sigma(G_2/\PP_2)-2\ge 3.$$
This is a contradiction, showing that $G$ is triangle-free.
\end{proof}

\section{Concluding remarks}

Our proof Theorem~\ref{thm-main} uses the result of Kostochka and Yancey~\cite{koyanore} to deal with the basic
case of planar graphs (Observation~\ref{obs-koyan}).  This is mostly for convenience---at all the places where
we use Observation~\ref{obs-koyan}, the graph in question is a specific composition of two exceptional graphs,
and we could make the argument self-contained by arguing that the graph resulting from this composition either
is exceptional or admits a nowhere-zero $\zz$-flow.  We chose to avoid doing this, as the arguments are tedious and
do not bring any new ideas.

Let us remark that Theorem~\ref{thm-li} can be proved similarly to Theorem~\ref{thm-smallcex}.  For simplicity, let us
outline the proof of a weaker bound $|E(G)|\le 4|V(G)|-7$, which holds also for $G=K_2$:
Defining $\sigma'(G)=4|V(G)|-|E(G)|$ and saying that $(G,\beta)$ is a smallest
counterexample if it has the smallest number of vertices among flow-critical $\zz$-bordered graphs with $\sigma'(G)<7$,
the analogue of Lemma~\ref{lemma-addprime} states that for every partition $\PP$ with at least two parts,
we have $\sigma'(G/\PP)<7-3l(\PP)$.  Following the argument of Theorem~\ref{thm-smallcex}, we show that
$G$ is $5$-edge-connected and the only $(\le\!6)$-cuts are vertex neighborhoods; and that $G$ does not
contain vertices of degree five, and thus it is 6-edge-connected.  This is a contradiction, since 6-edge-connected
graphs are $\zz$-connected~\cite{ltwz}.

It is also natural to ask the density question for $\mathbb{Z}_4$-flow-critical (or equivalently for $\mathbb{Z}_2^2$-flow-critical) graphs.
Note that by dualizing the Four Color Theorem, there are no planar $\mathbb{Z}_4$-flow-critical graphs different from $K_2$.
Not much is known about 5-critical graphs drawn on other surfaces (in particular, we do not know whether 4-colorability is decidable in polynomial time for graphs drawn in any fixed surface
of non-zero genus).  Moreover, on orientable surfaces of non-zero genus, the duality only holds in one direction (if a graph is 4-colorable, its dual has a nowhere-zero $\mathbb{Z}_4$-flow,
but the converse is not necessarily true); and consequently, duals of 5-critical graphs are not necessarily $\mathbb{Z}_4$-flow-critical
and duals of $\mathbb{Z}_4$-flow-critical graphs are not necessarily 5-critical.  And indeed, the two concepts seem to behave
quite differently on surfaces; for example, on any fixed surface of non-zero genus, there are 5-critical graphs of arbitrarily
large edgewidth~\cite{Fisk78}, but it has been conjectured that all 3-regular graphs whose duals have sufficiently large edgewidth
admit nowhere-zero $\mathbb{Z}_4$-flows~\cite{albertson2010grunbaum}.  Kochol~\cite{kochdis} gave, for any orientable surface $\Sigma$
of genus at least five, infinitely many examples of 3-regular 3-edge-connected graphs with no nowhere-zero $\mathbb{Z}_4$-flows drawn in $\Sigma$
such that their dual has edgewidth at least three; however, each of these graphs contains a specific graph with 66 vertices as a contraction,
and thus they do not provide an infinite family of $\mathbb{Z}_4$-flow-critical graphs.  
This leaves us without a natural candidate for a bound depending on the genus.
\begin{problem}
What is the smallest constant $c$ such that for every surface $\Sigma$, there exists $a_\Sigma$ such that every $\mathbb{Z}_4$-flow-critical graph $G$ drawn in $\Sigma$
satisfies $|E(G)|\le c|V(G)|+a_\Sigma$?
\end{problem}

With regards to the bounds independent of the genus, the situation is relatively clear in the bordered setting.
Consider the $\mathbb{Z}_2^2$-bordered graph $(K_{2,n-2},\beta)$ with $n\ge 5$ odd and with $\beta$ assigning
$(0,1)$ to all vertices except for one of the vertices of degree $n-2$, to which it assigns $(0,0)$.  This $\mathbb{Z}_2^2$-bordered graph is
easily seen to be flow-critical, it is planar and has $2n-4$ edges.  Similarly, $K_{2,n-2}$ is flow-critical with a properly chosen
$\mathbb{Z}_4$-boundary.  On the other hand, using the fact that every $2$-edge-connected graph that is one edge short from
having two spanning trees is both $\mathbb{Z}_2^2$- and $\mathbb{Z}_4$-connected~\cite{catlin1996graphs,lai1999extending}
and an argument similar to the proof of Theorem~\ref{thm-smallcex}, it is easy to see that
every $\mathbb{Z}_2^2$-bordered or $\mathbb{Z}_4$-bordered flow-critical graph $(G,\beta)$ with $G\neq K_2$ satisfies $|E(G)|\le 2|V(G)|-4$.
Let us remark that while a graph has a nowhere-zero $\mathbb{Z}_4$-flow if and only if it has a nowhere-zero $\mathbb{Z}_2^2$-flow,
this is not the case in the bordered setting~\cite{huvsek2020group}, and thus it is not a priori obvious that the answers
for $\mathbb{Z}_2^2$-bordered and $\mathbb{Z}_4$-bordered graphs should be the same.

However, in the non-bordered setting, the behavior seems to be quite different.  In particular, $\mathbb{Z}_2^2$-flow-critical graphs
without boundary cannot contain vertices of degree two, and thus the example of $K_{2,n-2}$ does not apply.
\begin{problem}
What is the smallest constant $c$ such that every $\mathbb{Z}_2^2$-flow-critical graph $G$ satisfies $|E(G)|\le c|V(G)|+O(1)$?
\end{problem}

Flower snarks give examples of arbitrarily large 3-regular $\mathbb{Z}_2^2$-flow-critical graphs~\cite{flower}, and thus $c\ge 3/2$.

\section*{Acknowledgement}

We would like to thank Jiaao Li for useful comments on the paper, and in
particular for helping us improve the discussion regarding the flow-critical
$\mathbb{Z}_2^2$- and $\mathbb{Z}_4$-bordered graphs.

\bibliographystyle{acm}
\bibliography{flowcrit.bib}

\end{document}